\documentclass[reqno]{amsart}
\usepackage{amssymb}
\usepackage{hyperref}

\title[ ]
{Pseudo B-Fredholm Operators and Spectral Theory  }

\author[ A. TAJMOUATI, M. AMOUCH and M.KARMOUNI]
{  A. TAJMOUATI, M. AMOUCH, M.KARMOUNI}

\address{A. TAJMOUATI and M.KARMOUNI \newline
 Sidi Mohamed Ben Abdellah
 Univeristy, 
 Faculty of Sciences Dhar Al Mahraz, Laboratory of Mathematical Analysis and Applications Fez, Morocco.}
\email{abdelaziz.tajmouati@usmba.ac.ma}
\email{mohammed.karmouni@usmba.ac.ma}

\address{Mohamed AMOUCH \newline
Department of Mathematics
University Chouaib Doukkali,
Faculty of Sciences, Eljadida.
24000, Eljadida, Morocco.}
\email{mohamed.amouch@gmail.com}


\newtheorem{theorem}{Theorem}[section]

\newtheorem{remark}{Remark}
\newtheorem{lemma}{Lemma}[section]
\newtheorem{proposition}{Proposition}[section]
\newtheorem{corollary}{Corollary}[section]
\newtheorem{example}{Example}

\begin{document}

\maketitle

\begin{abstract}
 In this paper, we show that every pseudo B-Fredholm operator is a pseudo Fredholm operator. Afterwards,
we prove that the pseudo B-Weyl spectrum  is empty if and only if the pseudo B-Fredholm spectrum is empty. Also,
we study a symmetric difference between some parts of the spectrum.
\end{abstract}

\section{Introduction and Preliminaries}

Throughout, $X$ denotes a complex Banach space and $\mathcal{B}(X)$ denotes the Banach algebra of all bounded linear
operators on $X$, we denote by $T^*$,  $R(T)$, $ R^{\infty}(T)=\bigcap_{n\geq0}R(T^n)$,  $K(T)$, $ H_0(T)$,  $\rho(T)$, $\sigma_{ap}(T)$, $\sigma_{su}(T)$, $\sigma(T)$,
 respectively the adjoint, the range, the hyper-range, the analytic core, the quasinilpotent part, the resolvent set, the  approximate point spectrum, the surjectivity spectrum and  the spectrum of $T$.

Next, let $T\in\mathcal{B}(X)$, $T$ is said to have the single
valued extension property at $\lambda_{0}\in\mathbb{C}$ (SVEP) if
for every  open neighbourhood   $U\subseteq \mathbb{C}$ of
$\lambda_{0}$, the only  analytic function  $f: U\longrightarrow
X$ which satisfies
 the equation $(T-zI)f(z)=0$ for all $z\in U$ is the function $f\equiv 0$. $T$ is said to have the SVEP if $T$ has the SVEP for
 every $\lambda\in\mathbb{C}$. Obviously, every operator $T\in\mathcal{B}(X)$ has the SVEP at every $\lambda\in\rho(T)$, hence $T$ and $T^*$ have the SVEP at every point of the boundary  $\partial( \sigma(T))$ of the spectrum.

 A bounded linear operator is called an upper semi- Fredholm (resp, lower semi Fredholm) if $dim N(T)<\infty \mbox{ and  } R(T) \mbox{ closed }$
(resp, $codim R(T) <\infty$). $T$ is semi Fredholm if it is a lower or upper. The index of a semi Fredholm operator $T$ is defined by $ind(T)= dim N(T)- codim R(T)$.

$T$ is a Fredholm operator if is a lower and upper semi Fredholm, and is called a Weyl operator if it is a Fredholm of index zero.\\
The essential  and Weyl spectrum  of $T$  are closed and  defined  by : $$\sigma_{e}(T)=\{\lambda\in \mathbb{C}:\>\> T-\lambda \mbox{  is not a  Fredholm operator}\}$$
$$\sigma_{W}(T)=\{\lambda\in \mathbb{C}:\>\> T-\lambda\>\> \mbox{ is not  a Weyl operator}\}.$$

Recall that $T\in\mathcal{B}(X)$ is said to be Kato operator or
semi-regular,  if $R(T)$ is closed
and $N(T)\subseteq R^{\infty}(T)$. Denote by $\rho_{K}(T)$ :
$\rho_{K}(T)=\{\lambda\in\mathbb{C}: T-\lambda I\mbox{  is Kato }
\}$ the Kato resolvent  and
$\sigma_{K}(T)=\mathbb{C}\backslash\rho_{K}(T)$ the Kato spectrum
of $T$. It is well known that $\rho_{K}(T)$ is an open subset of
$\mathbb{C}$.\\\\
Let $T\in \mathcal{B}(X)$ such that $X=X_{1}\oplus X_{2}$,  $T=T_{1}\oplus T_{2}$ :\\

$T$ is a B-Fredholm operator if $T_{1}$ is Fredholm and $T_{2}$ is nilpotent. The  B-Fredholm spectrum defined by: $$\sigma_{BF}(T)=\{\lambda\in \mathbb{C}:\>\> T-\lambda \mbox{ is not  B-Fredholm} \}.$$
  This class of operators, introduced and studied by Berkani et al. in a series of papers
which extends the class of semi-Fredholm operators. In the beginning this class was defined by:
 An operator $T\in \mathcal{B}(X)$,  is said to be B-Fredholm,  if for some integer $n\geq0$ the range $R(T^n)$ is closed and $T_n$, the restriction of $T$ to $R(T^n)$ is a Fredholm operator. $T$ is said to be a B-Weyl operator if $T_n$ is a Fredholm operator of index zero which is also equivalent to  the fact that  $T_{1}$  is a Weyl operator and $T_{2}$  is nilpotent. The  B-Weyl spectrum defined by $$\sigma_{BW}(T)=\{\lambda\in \mathbb{C}:\>\> T-\lambda  \mbox{ is not  B-Weyl}\}.$$
 Note that, Berkani  gave the equivalence onto this two definitions of B-Fredholm operator, see \cite[Theorem 2.7]{B1}.
It is easily seen that every nilpotent operator, as well as any idempotent bounded operator, is B-Fredholm.\\\\
More recently, B-Fredholm and B-Weyl operators were generalized to pseudo B-Fredholm and pseudo B-Weyl \cite{BO}, \cite{ZZ}. Precisely,
$T$ is a pseudo B-Fredholm operator if $T_{1}$ is a Fredholm  operator and $T_{2}$ is a quasi-nilpotent operators. The  pseudo B-Fredholm spectrum is defined by $$\sigma_{pBF}(T)=\{\lambda\in \mathbb{C}:\>\> T-\lambda \mbox{ is  not  pseudo B-Fredholm} \}.$$
An operator $T$ is a pseudo B-Weyl operator if $T_{1}$  is a Weyl operator and $T_{2}$ is a  quasi-nilpotent operator. The  pseudo B-Weyl spectrum defined by $$\sigma_{pBW}(T)=\{\lambda\in \mathbb{C}:\>\> T-\lambda \mbox{ is  not   pseudo B-Weyl} \}.$$
$\sigma_{pBW}(T)$ and $\sigma_{pBF}(T)$ is not necessarily non empty.
For example, the quasi nilpotent operator has empty pseudo B-Weyl  and
 pseudo B-Fredholm spectrum. Evidently $\sigma_{pBF}(T)\subset \sigma_{pBW}(T)\subset\sigma(T)$.\\\\
$T$ is a  pseudo-Fredholm operator (or admit generalized  Kato decomposition) if $T_{1}$ is Kato operator  and $T_{2}$ is quasi-nilpotent. The  pseudo-Fredholm spectrum defined by $$\sigma_{GK}(T)=\{\lambda\in \mathbb{C}:\>\> T-\lambda \mbox{ is not  a pseudo -Fredholm} \}$$ Denote by $\rho_{GK}(T)$ :
$\rho_{GK}(T)=\{\lambda\in\mathbb{C}: T-\lambda I\mbox{  is pseudo -Fredholm }\}$.
If we assume in the
definition above that $T_2$ is nilpotent, $T$ is said to be quasi-Fredholm (or admit a Kato  decomposition or  Kato type).
The  quasi-Fredholm spectrum defined by: $$\sigma_{tK}(T)=\{\lambda\in \mathbb{C}:\>\> T-\lambda \mbox{ is not   quasi -Fredholm} \}$$
 The Operators which admit a generalized Kato decomposition was originally introduced by M.Mbekhta \cite{Mb1} in the  Hilbert spaces as a generalization of quasi-Fredholm operators have been introduced by J.P.Labrousse \cite{lab} and the semi-Fredholm operators.\\\\
 In \cite{B1}, Berkani  showed that a B-Fredholm operator is a quasi-Fredholm, see \cite[Proposition 4]{B1}. This result lead to ask the following question: If  every pseudo B-Fredholm operator is a pseudo-Fredholm operator ?\\

We organize our paper in the following way: In the next section we prove  that  every pseudo B-Fredholm operator is a pseudo-Fredholm operator.  Also we study the relationships between the class of  pseudo B-Fredholm  and  other class of operator.
 In section 3, we shall study the component of pseudo B-Fredholm resolvent $\rho_{pBF}(T)$,
  to obtain a classification of the components
 by using the constancy of the subspaces quasi-nilpotent part and analytic core, some applications are  also given. Finally, in section 4, we  show that
the symmetric difference $\sigma_{K}(T)\Delta\sigma_{pBF}(T)$   is at most countable.

\section{ The class of Pseudo B-Fredholm Operators }

In the following theorem we prove that every  pseudo B-Fredholm operator is pseudo Fredholm.
\begin{theorem}\label{bb}
Let $T\in\mathcal{B}(X)$. If  $T$ is pseudo B-Fredholm, then $T$  is pseudo Fredholm.
\end{theorem}
\begin{proof}
Let $T\in\mathcal{B}(X)$. If $T$ is pseudo B-Fredholm operator, then there exists a subsets $M$ and $N$  of $X$ such that $X=M\oplus N$ and $T=T_1\oplus T_2$ with  $T_1=T_{|M}$ is a Fredholm operator and $T_2=T_{|N}$ is a quasi nilpotent. Since
$T_1$ is  Fredholm then $T$ admits a Kato decomposition, hence  there exists  $M'$, $M''$   closed subsets of $M$ such that $M=M'\oplus M''$,  $T_1=T'_1\oplus T^{''}_1$ with $T^{'}_1=T_{1 |M'}$ is a Kato operator and $T^{''}_1=T_{1 |M''}'$ is  nilpotent. Then $X=M'\oplus M^{''}\oplus N$, and $T=S\oplus R$ where $S=T'_1$ is a Kato operator and $R=T''_1\oplus T_2$ is a quasi-nilpotent operator, hence $T$ is a pseudo Fredholm operator.
\end{proof}

\noindent The following example shows that the class of pseudo  B-Fredholm operator is a proper
subclass of pseudo Fredholm operator.
\begin{example}
 Consider  the example given by  M\"{u}ller  in \cite{Mu}\\
Let $H$ be the Hilbert space with an orthonormal basis $(e_{i,j})$, where $i$ and $j$ are integers such that $ij\leq 0$.
Define operator $T\in \mathcal{B}(H)$ by :

$$
Te_{i,j}=
\left\{
\begin{array}{rl}
0 &  \mbox{ if }  i=0, j>0  \\
e_{i+1,j}  & \mbox{ Otherewise}
\end{array}
\right.
$$
We have $N(T)=\displaystyle\bigvee_{j>0}\{e_{0,j}\}\subset R^{\infty}(T)$ and $R(T)$ is closed, then $T$ is a Kato operator but $T$ is not a Fredholm operator, since $dim N(T)=\infty$.\\
Let $Q$ a quasinilpotent  operator  in $H$ which is not nilpotent  and no commute with $T$,  then $S=T\oplus Q$ is a pseudo Fredholm operator but is not pseudo B-Fredholm operator, hence the class of pseudo  B-Fredholm operator is a proper subclass of pseudo Fredholm operator.
\end{example}
\begin{remark}
In \cite[Remark 2.5]{ZZ} and \cite[Proposition 1.2  ]{W}, If $T$ is a bilateral shift on $l^2(\mathbb{N})$, we have :
\begin{enumerate}
  \item $T$ is  pseudo B-Weyl if and only if  $T$ is Weyl or $T$ is quasi-nilpotent operator.
  \item $T$  is  pseudo Fredholm if and only if  $T$ is semi-regular  or $T$ is quasi-nilpotent operator.
\end{enumerate}
By the same argument we can prove :
\begin{enumerate}
 \item $T$ is  pseudo B-Fredholm if and only if  $T$ is Fredholm or $T$ is quasi-nilpotent operator.
 \item $T$ is  generalized Drazin  if and only if  $T$ is invertible  or $T$ is quasi-nilpotent operator.
\end{enumerate}
\end{remark}

\begin{corollary}
Let $T\in \mathcal{B}(X)$. Then
\begin{center}
    $\sigma_{GK}(T)\subset\sigma_{pBF}(T)\subset\sigma_{pBW}(T)$
\end{center}
\end{corollary}

\begin{lemma}\cite{MO}\label{2}
Let $T\in \mathcal{B}(X)$ and let $G$ a connected component of $\rho_{K}(T)$. Then 
\begin{center}
    $G\cap \rho(T)\neq\varnothing\Longrightarrow  G\subset \rho(T)$
\end{center}
\end{lemma}
\begin{lemma}\label{aaa}\cite{W}
Let $T\in \mathcal{B}(X)$.\\
 $\rho_{GK}(T)\backslash \rho_{K}(T)$ is at most countable
\end{lemma}
Since $\rho_{pBF}(T)\setminus \rho_{K}(T)\subset \rho_{GK}(T)\backslash \rho_{K}(T)$, we can easily obtain that:

\begin{corollary}\label{1}
Let $T\in \mathcal{B}(X)$.\\
$\rho_{pBF}(T)\setminus \rho_{K}(T)$ is at most countable.
\end{corollary}

\begin{proposition}\label{ccc}
Let $T\in \mathcal{B}(X)$. Then  the following statements are equivalent  :
\begin{enumerate}

 \item $\sigma_{pBF}(T)$ is at most countable
 \item $\sigma_{pBW}(T)$ is at most countable
 \item $\sigma(T)$ is at most countable
 \end{enumerate}

\end{proposition}
\begin{proof}
 $1)\Longrightarrow 3)$ Suppose that $\sigma_{pBF}(T)$ is at most countable then $\rho_{pBF}(T)$ is connexe, by corollary \ref{1} $\rho_{pBF}(T)\setminus \rho_{K}(T)$ is at most countable. Hence $\rho_{K}(T)=\rho_{pBF}(T)\setminus(\rho_{pBF}(T)\setminus \rho_{K}(T))$ is connexe.  By lemma \ref{2} $\sigma(T)=\sigma_{K}(T)$. Therefore $\sigma(T)=\sigma_{pBF}(T)\cup(\rho_{pBF}(T)\setminus \rho_{K}(T))$ is at most countable.\\
 $3)\Longrightarrow 1)$ Obvious.\\
 $2)\Longrightarrow 3)$  If $\sigma_{pBW}(T)$ is at most countable then $\rho_{pBW}(T)$ is connexe,  since every pseudo B-Weyl operator is a pseudo B-Fredholm operator by corollary \ref{1} $\rho_{pBW}(T)\setminus \rho_{K}(T)$ is at most countable. Hence $\rho_{K}(T)=\rho_{pBW}(T)\setminus(\rho_{pBW}(T)\setminus \rho_{K}(T))$ is connexe.  By lemma \ref{2} $\sigma(T)=\sigma_{K}(T)$. Therefore $\sigma(T)=\sigma_{pBW}(T)\cup(\rho_{pBW}(T)\setminus \rho_{K}(T))$ is at most countable.\\
 $3)\Longrightarrow 2)$ Obvious.
\end{proof}
\begin{corollary}
Let $T\in \mathcal{B}(X)$, if $\sigma_{GK}(T)$ is at most countable. Then:\\
$T$ is a spectral operator if and only if $T$ is similar to a paranormal operator.
\end{corollary}

\begin{proof}
See \cite[Theoerem 2.4 and Corollary 2.5]{MMB}
\end{proof}

Let $T\in\mathcal{B}(X)$. The operator range topology on $R(T)$ is the topology induced by the norm $||.||_T$
 defined by $||y||_T:= \displaystyle \inf_{x\in X}\{||x|| : y = Tx\}.$\\
For a detailed discussion of operator ranges and their topology we refer the reader to \cite{G1}.\\
$T$ is said to have uniform descent  for $n\geq d$ if $R(T)+ N(T^n)=R(T)+N(T^d)$ for $n\geq d$. If in addition, $R(T^n)$ is closed in the operator range topology of $R(T^d)$ for $n\geq d$, then $T$ is said to have topological uniform descent (TUD for brevity )  for $n\geq d$. The topological  uniform descent spectrum :
$$\sigma_{ud}(T)=\{\lambda\in\mathbb{C}, T-\lambda \mbox{ does not have TUD}\}$$
Let $T\in\mathcal{B}(X)$, the ascent of $T$ is defined by $a(T)=min\{p\in\mathbb{N}: N(T^p)=N(T^{p+1})\}$, if such $p$ does not exist we let $a(T)=\infty$. Analogously the descent of $T$ is $d(T)=min\{q\in\mathbb{N}: R(T^q)=R(T^{q+1})\}$, if such $q$ does not exist we let $d(T)=\infty$ \cite{LT}. It is well known that
if both $a(T)$ and $d(T)$ are finite then $a(T)=d(T)$ and we have the decomposition $X=R(T^p)\oplus N(T^p)$ where $p=a(T)=d(T)$. The  descent and ascent spectra  of $T\in \mathcal{B}(X)$ are defined by :  $$\sigma_{des}(T)=\{\lambda\in\mathbb{C},\,\, T-\lambda \mbox{  has  not  finite   descent}\}$$
    $$\sigma_{ac}(T)=\{\lambda\in\mathbb{C},\,\, T-\lambda \mbox{ has  not  finite ascent }\}.$$

 On the other hand, a bounded operator $T\in\mathcal{B}(X)$ is said to be a Drazin invertible if there exists a positive integer $k$ and an operator  $S\in\mathcal{B}(X)$ such that  $$ST=TS, \,\,\,T^{k+1}S=T^k\,\, \,\,and\,\,  S^2T=S.$$
This  is also equivalent to  the fact that $T=T_1\oplus T_2$; where $T_1$ is invertible  and $T_2$ is nilpotent.
Recall that an operator $T$ is Drazin invertible if it has a finite ascent and descent.
The concept of Drazin invertible operators has been generalized by Koliha \cite{K}. In fact $T\in \mathcal{B}(X)$ is generalized Drazin invertible if and only if $0\notin acc\sigma(T)$ the set of all  points of accumulation of $\sigma(T)$, which is also equivalent to the fact that $T=T_1\oplus T_2$  where $T_1$ is invertible  and $T_2$ is quasinilpotent.
The  Drazin and generalized Drazin   spectra of $T\in \mathcal{B}(X)$ are defined by :
$$\sigma_{D}(T)=\{\lambda\in\mathbb{C},\,\, T-\lambda   \mbox{ is  not  Drazin invertible}\}$$
    $$\sigma_{gD}(T)=\{\lambda\in\mathbb{C},\,\, T-\lambda  \mbox{ is  not  generalized  Drazin} \}$$
    
We denote by, $\sigma^{e}_{des}(T)$, $\sigma_{rGD}(T)$ and  $\sigma_{lGD}(T)$ respectively the essential descent, right generalized Drazin and left generalized Drazin spectra of $T$.
According to corollary \ref{ccc}, \cite[Theorem 3.3]{W} and  \cite[corollary 3.4]{QZZZ}, we have the following:
\begin{corollary}
Let $T\in \mathcal{B}(X)$. Then  the following statements are equivalent
\begin{enumerate}
 \item $\sigma(T)$ is at most countable;
 \item $\sigma_{pBF}(T)$ is at most countable;
 \item $\sigma_{pBW}(T)$ is at most countable;
 \item $\sigma_{ud}(T)$ is at most countable;
 \item $\sigma_{GK}(T)$ is at most countable;
 \item $\sigma_{GD}(T)$ is at most countable;
 \item $\sigma_{lGD}(T)$ is at most countable;
 \item $\sigma_{rGD}(T)$ is at most countable;
 \item $\sigma_{D}(T)$ is at most countable;
 \item $\sigma_{K}(T)$ is at most countable;
 \item $\sigma_{BF}(T)$ is at most countable;
 \item $\sigma_{BW}(T)$ is at most countable;
 \item $\sigma_{des}(T)$ is at most countable;
 \item $\sigma^{e}_{des}(T)$ is at most countable;
 \end{enumerate}

 \end{corollary}


In \cite{QHH}, the showed that an operator with TUD for $n\geq d$, $K(T)=R^{\infty}(T)$ and $\overline{H_0(V)}= \overline{N^{\infty}(V)}$.\\
For a pseudo B-Fredholm operator, these properties do not necessarily hold. Indeed:
let $X$ be the Banach space of continuous functions on $[0, 1]$, denoted  by $\mathcal{C}([0,1])$, provided with the infinity norm. We define by $V$, the Volterra operator, $X$ by : $$Vf(x):=\int_{0}^{x} f(x) \, \mathrm{d}x$$.

$V$ is injective and quasi-nilpotent. In addition, $N^{\infty}(V)=\{0\}$, $K(V)=\{0\}$ and we have $R^{\infty}(V)=\{f\in C^{\infty}[0,1]:\,\, f^{(n)}(0)=0,\,\, n\in\mathbb{N}\}$, thus $R^{\infty}(V)$ is not closed. Hence:
\begin{enumerate}
  \item $K(V)\neq R^{\infty}(V)$
  \item $\overline{H_0(V)}\neq\overline{N^{\infty}(V)}$
\end{enumerate}
  Note that $V$ is a compact operator, then $R(V)$ is not closed.
\begin{theorem}
There exists a pseudo B-Fredholm operator  $T$ such that :
\begin{enumerate}
  \item $K(T)\neq R^{\infty}(T)$,
  \item $\overline{H_0(T)}\neq \overline{N^{\infty}(T)}$,
  \item $R(T)$ is not closed.
\end{enumerate}
\end{theorem}
\begin{proposition}\label{aa}
Let $T\in \mathcal{B}(X)$. Then  the following statements are equivalent
\begin{enumerate}

 \item $\sigma_{pBF}(T)$ is empty
 \item $\sigma_{pBW}(T)$ is empty
 \item $\sigma_{GK}(T)$ is empty
 \item $\sigma(T)$ is finite
 \end{enumerate}
\end{proposition}
\begin{proof}
$3)\Longleftrightarrow 4)$ see \cite[Theorem 3.3]{W}. \\
$1)\Longrightarrow 4)$ If $\sigma_{pBF}(T)$ is empty then $\sigma(T)=\rho_{pBW}(T)\setminus \rho_{K}(T)$. By corollary \ref{1} $\rho_{pBW}(T)\setminus \rho_{K}(T)$  is at most countable  and  this set is bounded, hence  it is finite.\\
$4)\Longrightarrow 1)$ Suppose that $\sigma(T)$ is finite then for all $\lambda\in \sigma(T)$ is isolated, then $X=H_{0}(T-\lambda_{0})\oplus K(T-\lambda_{0})$, \cite[Theorem 4]{Schh}  $(T-\lambda_{0})_{\shortmid H_{0}(T-\lambda_{0})}$ is quasi nilpotent  and  $(T-\lambda_{0})_{\shortmid K(T-\lambda_{0})}$ is surjective, hence  $(T-\lambda_{0})_{\shortmid K(T-\lambda_{0})}$ is Fredholm. Indeed, $\lambda_{0}$ is an  isolated point, then $T$ has the SVEP at $\lambda_{0}$, hence  $(T-\lambda_{0})_{\shortmid K(T-\lambda_{0})}$ has the SVEP  at 0 and  it is surjective by \cite[corollary 2.24]{Aie} $(T-\lambda_{0})_{\shortmid K(T-\lambda_{0})}$ is bijective   \\
$2)\Longleftrightarrow 4)$ similar to $1)\Longleftrightarrow 4)$.
\end{proof}
A bounded operator $T\in\mathcal{B}(X)$ is said to be a Riesz operator  if $T-\lambda I$ is
  a Fredholm operator for every $\lambda\in \mathbb{C}\backslash\{0\}$.
\begin{corollary}
Let $T\in\mathcal{B}(X)$ a Riesz operator, then  the following statements are equivalent
\begin{enumerate}

 \item $\sigma_{pBF}(T)$ is empty,
 \item $\sigma_{pBW}(T)$ is empty,
 \item $\sigma_{GK}(T)$ is empty,
 \item $\sigma(T)$ is finite,
 \item $K(T)$ is closed,
 \item $K(T^*)$ is closed,
 \item $K(T)$ is finite-dimensional,
 \item $K(T-\lambda)$ is closed for all $\lambda\in \mathbb{C}$,
 \item $codim H_{0}(T)<\infty$,
 \item $codim H_{0}(T^*)<\infty$,
 \item $T=Q+F$, with $Q, F\in\mathcal{B}(X)$, $QF=FQ=0$,  $\sigma(Q)=\{0\}$ and $F$ is a finite rank operator.
 \end{enumerate}
\end{corollary}
\begin{proof}
Direct consequence of Proposition \ref{aa} and \cite[Theorem 2.3]{W1} and \cite[Corollary 9]{MMN}
\end{proof}

In the following, we will prove that if $T$  is with finite descent, then $T$ is pseudo B-Fredholm if and only if $T$
is a B-Fredholm operator.
\begin{proposition}
Let $T\in\mathcal{B}(X)$ with finite descent. Then $T$ is a pseudo B-Fredholm if and only if $T$ is a B-Fredholm.
\end{proposition}
\begin{proof}
Obviously if $T$ is B-Fredholm then $T$ is pseudo B-Fredholm.\\
If $T$ is a pseudo B-Fredholm then $T=T_1\oplus T_2$ with $T_1$ is Fredholm operator and $T_2$ is quasinilpotent. Since $T$ has finite descent then $T_1$ and $T_2$ have finite descent, we have $T_2$ is quasinilpotent with finite descent implies that is a nilpotent operator. Thus $T$ is a B-Fredholm operator.
\end{proof}
In the following,  we show that an operator  with dense range is pseudo Fredholm  if and only if it is  a semi regular.
\begin{proposition}
Let $T\in \mathcal{B}(X)$. If $T$ is  with dense range, then :
\begin{center}
$T$ is a pseudo Fredholm if and only if $T$ is semi regular.
\end{center}

 \end{proposition}
 \begin{proof}
Every semi regular operator is a pseudo Fredholm.
Conversely, if $T$ admits a GKD, there exists a pair of $T$-invariant closed
subspaces $(M,N)$ such that $X=M\oplus N$, the restriction
$T_{\shortmid M}$ is semi-regular, and $T_{\shortmid N}$ is
quasinilpotent. $T$ has dense range give $\overline{R(T)}=X$  $\Rightarrow N(T^*)=\{0\}$ then $T^*$ have the SVEP at $0$. According to \cite[Theorem 3.15]{Aie}, we have $K(T)=M$. Since $M$  is closed and $T_{\shortmid M}$ is semi-regular (then $M=X$), therefore $T$ is semi-regular.
\end{proof}
\begin{corollary}
Let $T\in \mathcal{B}(X)$, with dense range. Then :
\begin{center}
 $T$ is a pseudo B-Fredholm $\Rightarrow$  $T$ is semi regular.
\end{center}
In particular,  $T$ is a generalized Drasin invertible $\Rightarrow$  $T$ is semi regular.

\end{corollary}


\section{ Classification Of The Components Of Pseudo B-Fredholm Resolvent }

\begin{lemma}
Let $T\in\mathcal{B}(X)$ a pseudo B-Fredholm, then there exists $\varepsilon>0$ such that for all  $|\lambda|<\varepsilon$, we have:
\begin{enumerate}
  \item $K(T-\lambda)+H_{0}(T-\lambda)=K(T)+H_{0}(T)$.
  \item $K(T-\lambda)\cap \overline{H_{0}(T-\lambda)}=K(T) \cap \overline{H_{0}(T)}$.
\end{enumerate}
\end{lemma}
\begin{proof}
By Theorem \ref{bb}, $T$ is a pseudo Fredholm operator, hence we conclude by \cite[Theorem 4.2]{W} the result.
\end{proof}
 The pseudo B-Fredholm  resolvent set is defined as $\rho_{pBF}(T)=\mathbb{C}\backslash\sigma_{pBF}(T).$
\begin{corollary}\label{ab}
Let $T\in\mathcal{B}(X)$ a pseudo B-Fredholm operator, then the mappings \\

$\lambda\longrightarrow K(T-\lambda)+H_{0}(T-\lambda)$, $\lambda\longrightarrow K(T-\lambda)\cap \overline{H_{0}(T-\lambda)}$
are constant on the components of $\rho_{pBF}(T)$.
\end{corollary}
\begin{lemma}\label{123}
Let $T$  a pseudo B-Fredholm operator. Then the following statements are equivalent:
\begin{enumerate}
  \item $T$ has the SVEP at $0$,
  \item $\sigma_{ap}(T)$ does not a cluster at $0$.
\end{enumerate}
\end{lemma}
\begin{proof}
Without loss of generality,  we can assume that $\lambda_0=0$.\\
$2)\Rightarrow 1)$  See \cite {Aie}.\\
$1)\Rightarrow 2)$  Suppose that $T$ is a pseudo B-Fredholm operator, then there exists two closed $T$-invariant subspaces $ X_{1} \>\>,  X_{2} \subset X$ such that $ X=X_{1}\oplus X_{2}$, $T_{\shortmid X_{1}}$  is Fredholm,  $T_{\shortmid X_{2}}$ is quasi-nilpotent  and $T=T_{\shortmid X_{1}} \oplus T_{\shortmid X_{2}}$.  Since $T_{\shortmid X_{1}}$  is Fredholm,  then $T_{\shortmid X_{1}}$ is of   Kato type by \cite[Theorem 2.2]{AR}   there exists a constant $\varepsilon >0 $ such that for all $\lambda\in D^*(0,\varepsilon)$,  $\lambda I -T$ is bounded below. Since $T_{\shortmid X_{2}}$ is quasi-nilpotent,  $\lambda I-T$ is bounded below for all $\lambda\neq 0$.  Hence  $\lambda I -T$ is  bounded below  for all $\lambda\in D^*(0,\varepsilon)$. Therefore  $\sigma_{ap}(T)$ does not cluster at $\lambda_0$.

\end{proof}
By duality we have :
\begin{lemma}\label{AA}
Let $T$  a pseudo B-Fredholm operator. Then the following statements are equivalent:
\begin{enumerate}
  \item $T^*$ has the SVEP at $0$,
  \item $\sigma_{su}(T)$ does not a cluster at $0$.
\end{enumerate}
\end{lemma}

\begin{theorem}\label{abb}
Let $T\in \mathcal{B}(X)$ and  $\Omega$ a component of $\rho_{pBF}(T)$. Then the following alternative holds:
\begin{enumerate}
  \item $T$ has the SVEP for every point of $\Omega$. In this case, $\sigma_{ap}(T)$ does not have limit points in $\Omega$, every point of $\Omega$ is not an eigenvalue of $T$ execpt a subset of $\Omega$ which consists of at most countably many isolated points.
  \item $T$ has the SVEP at no point of $\Omega$. In this case, every point of $\Omega$ is an eigenvalue of $T$.
\end{enumerate}
\end{theorem}
\begin{proof}
 $1)$ Assume that $T$ has the SVEP at $\lambda_{0}\in \Omega$. By \cite[Theorem 3.14]{Aie} we have $K(T-\lambda_{0}) \cap H_{0}(T-\lambda_{0})=K(T-\lambda_{0})\cap\overline{H_{0}(T-\lambda_{0})}=\{0\}$. According to corollary \ref{ab}, we have $K(T-\lambda_{0})\cap\overline{H_{0}(T-\lambda_{0})}=\{0\}=K(T-\lambda)\cap\overline{H_{0}(T-\lambda)}=\{0\}$ for all $\lambda\in\Omega$. Hence $K(T-\lambda)+\overline{H_{0}(T-\lambda)}=\{0\}$ and therefor $T$ has the SVEP at every $\lambda\in \Omega$. By Lemma \ref{123}, $\sigma_{ap}(T)$ does not cluster at any  $\lambda\in\Omega$. Consequently every point of $\Omega$ is not an eigenvalue of $T$ execpt a subset of $\Omega$ which consists of at most countably many isolated points.\\
$2)$ Suppose that $T$ has the SVEP at not point of $\Omega$. From \cite[Theorem 2.22]{Aie}, we have $N(T-\lambda) \neq \{0\}$, for all $\lambda \in \Omega$,  hence every point of $\Omega$ is an eigenvalue of $T$.
\end{proof}

\begin{theorem}\label{aab}
Let $T\in \mathcal{B}(X)$ and  $\Omega$ a component of $\rho_{pBF}(T)$. Then the following alternative holds:
\begin{enumerate}
  \item $T^*$ has the SVEP for every point of $\Omega$. In this case, $\sigma_{su}(T)$ does not have limit points in $\Omega$, every point of $\Omega$ is not  a deficiency value  of $T$ execpt a subset of $\Omega$ which consists of at most countably many isolated points.
  \item $T^*$ has the SVEP at no point of $\Omega$. In this case, every point of $\Omega$ is  a deficiency value of $T$.
\end{enumerate}
\end{theorem}

\begin{proof}
 $1)$ Assume that $T^*$ has the SVEP at $\lambda_{0}\in \Omega$, by \cite[Theorem 3.15]{Aie} we have $K(T-\lambda_{0})+H_{0}(T-\lambda_{0})=X$. According to corollary \ref{ab}, we have $K(T-\lambda_{0})+\overline{H_{0}(T-\lambda_{0})}= K(T-\lambda)+\overline{H_{0}(T-\lambda)}=X$ for all $\lambda\in\Omega$. Hence $K(T-\lambda)+\overline{H_{0}(T-\lambda)}=X$ and therefor $T$ has the SVEP at every $\lambda\in \Omega$. By lemma \ref{AA}, $\sigma_{su}(T)$ does not cluster at any  $\lambda\in\Omega$. Consequently every point of $\Omega$ is not a deficiency value of $T$ execpt a subset of $\Omega$ which consists of at most countably many isolated points.\\
$2)$ Suppose that $T^*$ has the SVEP at no point of $\Omega$. Assume that there exists a $\lambda_{0}\in \Omega$ such that $T-\lambda$ is surjective, then $T^*-\lambda_{0}$ is injective this  implies that  $T^*$ has the SVEP at $\lambda_{0}$. Contraduction and  hence every point of $\Omega$ is a deficiency value of $T$.
\end{proof}
\begin{remark}
We have $\sigma_{pBF}(T)\subset\sigma_{gD}(T)$, this inclusion is proper. Indeed: Consider the operator $T$ defined  in $l^2(\mathbb{N})$ by
$$T(x_{1},x_{2},....)=(0,x_{1},x_{2},...),~~~  T^*(x_{1},x_{2},....)=(x_{2},x_{3},...).$$
Let $B=T\oplus T^*$. Then $\sigma_{gD}(T)=\{\lambda\in \mathbb{C}; |\lambda|\leq 1\}$ and we have $0\notin\sigma_{pBF}(T)$. This shows that the inclusion $\sigma_{pBF}(T)\subset\sigma_{gD}(T)$ is proper.
\end{remark}
Next we obtain a condition on an operator such that its pseudo B-Fredholm spectrum  coincide with the generalized Drazin
spectrum.
\begin{theorem}
Suppose that $ T\in \mathcal{B}(X)$ and $\rho_{pBF}(T)$ has only one component.
Then $$\sigma_{pBF}(T)=\sigma_{gD}(T)$$
\end{theorem}
\begin{proof}
$\rho_{pBF}(T)$ has only one component, then $\rho_{pBF}(T)$ is the unique component. Since $T$ has the SVEP on $\rho(T)\subset\rho_{pBF}(T)$. By Theorem \ref{abb}, $T$ has the SVEP on $\rho_{pBF}(T)$.
Similar $T^*$  also has the SVEP on $\rho_{pBF}(T)$ by  Theorem \ref{aab}. ( This since $\rho(T^*)=\rho(T)\subset\rho_{pBF}(T)$).
From Lemma \ref{123}  and  Lemma \ref{AA}, $\sigma(T)$ does not cluster at any $\lambda\in \rho_{pBF}(T).$
Therefor $\rho_{pBF}(T)\subset iso\sigma(T)\cup\rho(T)=\rho_{gD}(T)$, hence $\rho_{pBF}(T)=\rho_{gD}(T)$.
\end{proof}

\section{ Symmetric difference for pseudo B-Fredholm spectrum }

\noindent Let in the following we give symmetric difference between
$\sigma_{pBF}(T)$ and other parts of the spectrum. Denoted by $\rho_{fK}(T)=\{\lambda\in\mathbb{C}, K(T-\lambda)\,\,\, is\,\,not\,\, closed\}$, $\sigma_{fK}(T)=\mathbb{C}\setminus\rho_{fK}(T)$ and  $\rho_{cr}(T)=\{\lambda\in\mathbb{C}, R(T-\lambda)\,\,\, is\,\, closed\}$,
$\sigma_{cr}(T)=\mathbb{C}\setminus\rho_{cr}(T)$ the Goldberg spectrum.
Most of the classes of operators, for example, in Fredholm theory, require that the
operators have closed ranges. Thus, it is natural to consider the closed-range spectrum or
Goldberg spectrum of an operator.
\begin{proposition}\label{asc}
If $\lambda\in \sigma_{*}(T)$ is non- isolated point then $\lambda\in\sigma_{pBF}(T)$, where $*\in\{fK, cr\}$.
\end{proposition}
\begin{proof}
Let $\lambda\in\sigma_{*}(T)$  an isolated point. Suppose that $T-\lambda$ is a pseudo B-Fredholm, by Lemma \ref{aaa}  there exists a constant $\varepsilon >0 $ such that for all $\lambda\in D^*(\lambda,\varepsilon)$,  $\lambda-  T$ is semi regular. Then $R(T-\mu)$  and $K(T-\mu)$ are closed for all $\mu\in D^*(\lambda,\varepsilon)$, then $\lambda$ is a isolated point of $\sigma_{*}(T)$, contradiction.
\end{proof}

\begin{corollary}
$\sigma_{*}(T)\backslash\sigma_{pBF}(T)$ is at most countable, where  $*\in\{fK, cr\}$.
\end{corollary}

\begin{proposition}
Let $T\in \mathcal{B}(X)$ such that   $\sigma_{cr}(T)=\sigma(T)$ and every $\lambda$ is non-isolated in $\sigma(T)$. Then
\begin{center}
    $\sigma(T)=\sigma_{cr}(T)=\sigma_{pBF}(T)=\sigma_{pBW}(T)=\sigma_{e}(T)=\sigma_{K}(T)=\sigma_{ap}(T)$
\end{center}
\end{proposition}

\begin{proof}
Since every $\lambda\in\sigma(T)=\sigma_{cr}(T)$ is non-isolated then by Proposition \ref{asc},  we have $\sigma(T)=\sigma_{cr}(T)\subseteq \sigma_{pBF}(T)\subseteq\sigma_{pBW}(T)\subseteq\sigma_{e}(T)\subseteq\sigma(T)$ and since
$\sigma(T)=\sigma_{cr}(T)\subseteq \sigma_{K}(T)\subseteq\sigma_{ap}(T)\subseteq\sigma(T)$,  we deduce the statement of the theorem.
\end{proof}

\begin{proposition}
The symmetric difference $\sigma_{K}(T)\Delta\sigma_{pBF}(T)$   is at most countable.
\end{proposition}
\begin{proof}
By corollary \ref{1},  $\sigma_{K}(T)\backslash\sigma_{pBF}(T)$ is at most countable.
We have $\sigma_{e}(T)\backslash\sigma_{K}(T)$ consists of at most countably many isolated points (see \cite[Theorem 1.65]{Aie} and
$\sigma_{pBF}(T)\backslash\sigma_{K}(T)\subseteq\sigma_{e}(T)\backslash\sigma_{K}(T)$, hence $\sigma_{pBF}(T)\backslash\sigma_{K}(T)$ is at most countable. Since\\ $$\sigma_{K}(T)\Delta\sigma_{pBF}(T)=(\sigma_{K}(T)\backslash\sigma_{pBF}(T))\bigcup (\sigma_{pBF}(T)\backslash\sigma_{K}(T))$$ Therefore $\sigma_{K}(T)\Delta\sigma_{pBF}(T)$   is at most countable.
\end{proof}


\end{document}